\documentclass[10pt]{amsart}

\usepackage[margin=1in]{geometry}

\usepackage{graphicx,epstopdf}
\usepackage[usenames,dvipsnames]{color}
\usepackage[colorlinks=true, pdfstartview=FitV, linkcolor=RoyalBlue,
           citecolor=ForestGreen, urlcolor=blue]{hyperref}

\usepackage{esint}

\title[Irregular free boundary points]{Classification of irregular free boundary points\\
for non-divergence type equations\\
with discontinuous coefficients}

\author{Serena Dipierro}
\address[Serena Dipierro]{Dipartimento di Matematica, Universit\`a degli studi di Milano,
Via Saldini 50, 20133 Milan, Italy}
\email{serena.dipierro@unimi.it}

\author{Aram Karakhanyan}
\address[Aram Karakhanyan]{Maxwell Institute for
Mathematical Sciences and School of Mathematics, University of
Edinburgh, James Clerk Maxwell Building, Peter Guthrie Tait Road,
Edinburgh EH9 3FD, United Kingdom}
\email{aram.karakhanyan@ed.ac.uk}

\author{Enrico Valdinoci}
\address[Enrico Valdinoci]{School of Mathematics and Statistics,
University of Melbourne,
813 Swanston Street, Parkville VIC 3010, Australia, 
and Istituto di Matematica Applicata e Tecnologie Informatiche,
Consiglio Nazionale delle Ricerche,
Via Ferrata 1, 27100 Pavia, Italy,
and Dipartimento di Matematica, Universit\`a degli studi di Milano,
Via Saldini 50, 20133 Milan, Italy}
\email{enrico@mat.uniroma3.it}

\usepackage{amsmath, amsthm, amssymb,amsfonts,%latexsym,verbatim%,amsbsy
}

\usepackage[usenames,dvipsnames]{color}

\usepackage{hyperref}

\makeatletter
\@addtoreset{equation}{section}
\makeatother

\theoremstyle{plain}
\newtheorem{theorem}{Theorem}[section]

\newtheorem{lemma}[theorem]{Lemma}
\newtheorem{prop}[theorem]{Proposition}

\theoremstyle{definition}
\newtheorem{remark}[theorem]{Remark}

\renewcommand{\tilde}{\widetilde}
\def\R{\mathbb R}
\def\Z{\mathbb Z}
\def\supp{\hbox{supp}}
\def\e{{\varepsilon}}
\def\ep{\epsilon}
\newcommand{\p}[1]{\partial_{#1}}

\newcommand{\fb}[1]{\partial\{{#1>0}\}}
\renewcommand\th{\theta}
\renewcommand\L{\mathcal L}
\renewcommand\S{\mathbb S}
\newcommand\na{\nabla}

\newcommand{\be}{\begin{equation}}
\newcommand{\ee}{\end{equation}}
\renewcommand{\div}{\operatorname{div}}

\begin{document}

\begin{abstract}
We provide an integral estimate for a non-divergence 
(non-variational) form second order 
elliptic equation $a_{ij}u_{ij}=u^p$, $u\ge 0$, $p\in[0, 1)$, 
with bounded discontinuous  coefficients $a_{ij}$ having small BMO norm. 
We consider the simplest discontinuity of the 
form~$x\otimes x|x|^{-2}$ at the origin. As an application 
we show that the free boundary corresponding to the 
obstacle problem (i.e. when~$p=0$) 
cannot be smooth at the points of discontinuity of~$a_{ij}(x)$. 

To implement our construction,
an integral estimate and a scale invariance will
provide the homogeneity of the blow-up sequences,
which then can be classified using ODE arguments. 
\end{abstract}


\maketitle

\section{Introduction}\label{sec:intro}

In this paper we consider the free boundary problem 
\be\label{pde-000}
\L v:=a_{ij}v_{ij}=v^p\ \; \hbox{in}\ B_1, \quad
v\ge 0,
\ee
with~$p\in(0,1)$. We will also deal with the case~$p=0$
using the notation that identifies~$v$ to the power zero with
the characteristic function~$\chi_{\{v>0\}}$.

Problems of this type often arise in real world phenomena.
For instance,
in the study of the spread of biological  populations one studies the problem
\begin{equation}\label{pde-0}
\div(a\na( u^m))+f(x)u+b\cdot \na (u^m)=0
\end{equation}
where $u:\R^n\to[0,+\infty)$ represents the density of the population,
$a:\R^n\to {\rm Mat}(n\times n)$ and~$b:\R^n\to\R^n$
represents a drift term.
Here, $m>1$, $a(x)$ is a positive definite matrix (with entries~$a_{ij}(x)$)
and~$f:\R^n\to\R$ takes into account
the 
influence of the environment on the population, see~\cite{S83}.

It is convenient to reformulate the problem in terms of the auxiliary function~$v:=u^m$ 
and write~\eqref{pde-0} as 
\[\div(a\na v)+f(x)v^{\frac1m}+b\cdot \na v=0.\]
Notice that this boils down to the equation in~\eqref{pde-000}
when~$m=1/p$, $f\equiv-1$ and~$b=(b_1,\dots,b_n)$ with~$b_i=\partial_ja_{ij}$.

The case in which~$a_{ij}$ is the identity matrix reduces of course to that of the Laplacian,
and, in general, a non-constant~$a_{ij}$ models a
heterogeneous medium in which the speed of diffusion
is different from one point to another.

Moreover, equations in non-divergence form
arise naturally from probabilistic considerations,
for instance, as the infinitesimal generators of anisotropic random
walks, see e.g. Section~2.1.3 in~\cite{C08}.

Furthermore, when~$a$ in~\eqref{pde-000} is the identity matrix,
the problem is related to the singular one in~\cite{AP86},
and as~$p\to0$ it recovers the exemplary
free boundary problem in~\cite{C77}. 
\medskip

One of the main distinctions in the field of partial differential
equations consists in the difference between equations
``in divergence form'' and those ``in
non-divergence form''.
While the first ones naturally admit a variational formulation
and can be dealt with by energy methods, the second ones
usually require different -- and perhaps more sophisticated -- techniques
(see e.g.~\cite{T82} for a detailed discussion), often in combination with viscosity methods.

We refer to~\cite{K07, C08} and the references therein
for throughout 
presentations of similarities and differences between
equations in divergence and non-divergence form.

A similar distinction between
divergence and non-divergence structure
occurs in the field of free boundary
problems. As a matter of fact,
free boundary problems whose partial differential
equation is in divergence form
often enjoy a special feature given by the so-called ``monotonicity
formulas'': namely, the energy functional, or a suitable
variational integral, possesses a natural monotonicity
property with respect to some geometric quantity (typically,
a functional defined on balls of radius~$r$ turns out to
be monotone in~$r$).

This type of monotonicity property is, in a sense, geometrically 
motivated, since
it may be seen somehow as an offspring of classical
monotonicity
formulas arising in the theory of minimal surfaces and geometric
flows. In addition, combined with the natural scaling of the problem,
a monotonicity formula is often very useful in proving
uniqueness of blow-up solutions, classification results
and regularity theorems.

Viceversa, problems which do not enjoy 
monotonicity formulas (or for which a monotonicity formula
is not known) may turn out to be considerably harder to deal with,
and proving (or disproving) a strong regularity theory
is a natural, important and often very challenging question
(see e.g.~\cite{CS-LIB, PSU12} for further discussions on monotonicity formulas).
\medskip

The study of free boundary
in discontinuous media is also a very active field
of research in itself, see in particular~\cite{T16}
for related problems involving
a fully nonlinear dead-core
problems, 
\cite{ALT16} for dead-core problems driven by the infinity Laplacian, and~\cite{PT16}
for cavity problems in
rough media. See also~\cite{Blan} 
for a case
in which the coefficients
belong to the space of vanishing mean oscillation.\medskip

Our objective in the present paper
is to study the behavior of the solution~$v$ of~\eqref{pde-000}
near the free boundary points~$x\in
\fb v$ at which the matrix~$a_{ij}(x)$ is discontinuous.  
A model example of this sort in~$2$D  is 
\be\label{pde-0BIS}
\Delta v+\e\left(\frac{x_1^2}{|x|^2}v_{22}-\frac{2x_1x_2}{|x|^2}v_{12}+\frac{x_2^2}{|x|^2}v_{11}\right)=v^p
\ee
where $\e$ is a small constant and $p\in[0,1)$
(here, we are using the standard
notation~$x=(x_1,x_2)\in\R^2$ and~$v_{ij}=\partial^2_{ij}v$).

One can also write equation~\eqref{pde-0BIS} in 
the equivalent form
\[
\div(a\na v)+b\cdot \na v=v^p
\]
where 
\begin{equation}\label{aijdef}
a(x):=
\begin{pmatrix}
1+\frac{\e x_2^2}{|x|^2} & -\frac{\e x_1x_2}{|x|^2}\\
-\frac{\e x_1x_2}{|x|^2}& 1 +\frac{\e x_1^2}{|x|^2}
\end{pmatrix}
\end{equation}
and 
\[b=(b^1, b^2), \quad b^j=-\sum_i\p i (a_{ij}), \quad  |b|\sim \frac1{|x|}.\]
We observe  that the quadratic form 
\[a_{ij}\xi_i\xi_j=|\xi|^2+\frac\e{|x|^2}\left((x_1\xi_2)^2+(x_2\xi_1)^2-2x_1x_2\xi_1\xi_2\right)=|\xi|^2+\frac\e{|x|^2}(x_1\xi_2-x_2\xi_1)^2\]
is positive definite and $a_{ij}$ are discontinuous at the origin.

More generally, we can assume that the diffusion matrix~$a$
has the form
\begin{equation}\label{PART}
a_{ij}(x)=h_{ij}(x)+b_{ij}(x)\end{equation}
where $h_{ij}$ is a homogeneous function of degree zero 
and for any point~$x_0\in\R^n$ we have that
$$|b_{ij}(x)-\delta_{ij}|\le \omega(|x-x_0|),$$ with 
\[\int_{0}^{\delta}\frac{\omega(t)}{t}dt<+\infty,\]
for some~$\delta>0$.
Roughly speaking, in~\eqref{PART}, the terms~$b_{ij}$
and~$h_{ij}$ represent the continuous and the discontinuous
parts of~$a_{ij}$, respectively.
\medskip

Throughout this paper we will assume that the operator satisfies the following 
conditions:
\begin{itemize}
\item[\bf (H1)] the entries of the matrix $a_{lm}$ 
are bounded measurable functions, and the matrix 
is uniformly elliptic,  i.e. there exist two positive 
constants $\lambda$ and~$\Lambda$ such that 
\[\lambda|\xi|^2\le a_{lm}(x)\xi_l\xi_m\le \Lambda|\xi|^2, 
\quad \forall x\in B_1,\]
\item[\bf (H2)] the coefficients $a_{lm}(x)$ have small BMO norm, namely
\[ \sup_{0<r\leq R}\sup_{x\in \R^n}\fint_{B_r(x)}\left|a_{lm}(y)-
\fint_{B_r(x)}a_{lm}\right|\,dy= \delta(R)<+\infty,\] 
where $\delta(R)>0$ is a small constant.
\item[\bf (H3)] the matrix $a_{ij}$ has at least one 
discontinuity at $x_0\in \R^n$ such that $a_{ij}(x)$ 
is rotational invariant at $x_0$ and homogeneous of degree zero. 
\end{itemize}
In this setting, the problem in~\eqref{pde-000}
admits a solution, as given by the following result:

\begin{theorem}\label{EX:TH}
Let $g\in W^{2,\infty}(B_1)\cap C(\overline{B_1})$, 
with~$g\ge 0$. Then, there exists  a
nonnegative function $v$ such 
that $v-g\in W^{2,q}(B_1)\cap W_0^{1, q}(B_1)$,
for some $1<q<+\infty$, and $v$ solves \eqref{pde-000}.
\end{theorem}

{F}rom the technical point of view, concerning the assumptions
on the coefficients~$a_{ij}$,
we notice that the function $x_ix_j|x|^{-2}\not\in VMO$ for any $i$ and~$j$.
However, if $\e$ is sufficiently small then $\bf (H2)$ holds with $\delta(R)\le C\e$,
where $C$ is a dimensional constant. Consequently, we can apply the $W^{2,q}$ estimates 
from Theorem~4.4 in~\cite{VMO} to establish the existence and optimal growth of the solutions.
As a matter of fact, setting
\begin{equation}\label{C0}
\beta=\frac2{1-p},\end{equation}
we can bound the growth from the free boundary
according to the following result
(see also Theorem~2 in~\cite{T16}):

\begin{theorem}\label{prop-dyadic:TH}
Let $v\ge 0$ be a bounded
weak solution of \eqref{pde-000} in $B_1$. 
Then there exists a constant $M>0$, depending 
on~$\|v\|_{L^\infty(B_1)}$, such that,
for each $\bar x\in B_{\frac12}\cap \fb v$
and any~$x\in B_{\frac14}(\bar x)$, it holds that~$v(x)\le M\,|x-\bar x|^\beta$.
\end{theorem}

We remark that the problem in~\eqref{pde-000} has a natural
scale invariance: for this, it is useful to define
$$v_r(x):= \frac{v(x_0+rx)}{r^\beta}$$ with~$\beta$ as in~\eqref{C0}.
We notice indeed that $v_r$ is also a solution of \eqref{pde-000}.
We will show that, up to a subsequence,
these blow-up functions approach a blow-up limit.

We say that $v$ is non-degenerate at $x_0\in \fb v$ if there exists a sequence of positive 
numbers $r_k\to 0$ such that the corresponding blow-up limit is not identically zero.

A cornerstone of our analysis is a uniform integral estimate.
The result that we obtain is the following:

\begin{theorem}\label{PRE:thm-2D}
Let $v$ be a strong solution 
of \eqref{pde-000} in $B_1\subset\R^2$,
with~$a_{ij}$ as in~\eqref{aijdef}.
Assume that 
$0\in \fb v$ and $v$ is non-degenerate at $0$. Then 
\begin{equation} \label{098io1}
\int_{B_{1/2}} \left(\beta\, \frac{v(x)}{|x|^\beta}  
-\frac{\p r v(x)}{|x|^{\beta-1}}\right)^2\,\frac{dx}{|x|^2}
\le \tilde C,\end{equation}
for some~$\tilde C>0$ depending on~$\|v\|_{L^\infty(B_1)}$.
\end{theorem}

In this framework, the integral estimate in~\eqref{098io1}, combined
with the scale invariance, implies that the blow-up limits
are homogeneous, as described in the following result:

\begin{theorem}\label{thm-2D}
Let $v$ be a strong solution 
of \eqref{pde-000} in $B_1$,
with~$a_{ij}$ as in~\eqref{aijdef}.
Assume that 
$0\in \fb v$ and $v$ is non-degenerate at $0$. Then any blow-up sequence at $0$ has a converging subsequence 
such that the limit is a homogeneous function of degree $\beta=\frac2{1-p}$. 
\end{theorem}

This result will in turn play a special role for the classification
of global solutions. Roughly speaking,
the homogeneity property, an appropriate use of polar coordinates
and explicit methods borrowed from the theory of ordinary differential
equations lead to a classification of solutions growing in a non-degenerate
way from a smooth free boundary. This classification
and the analysis of the blow-up limits
will be the main ingredients for the analysis
of irregular free boundary points,
as explained in the following result (compare also with Corollary~6.8
in~\cite{Blan}):

\begin{theorem}\label{DISCO}
Let~$n=2$, $\L$ be as in~\eqref{pde-000}
and~$a_{ij}$ as in~\eqref{aijdef}, with~$|\e|$ sufficiently small. 
Let~$v$ be 
a solution of~\eqref{pde-000} in~$B_1$ with~$p=0$. 
Assume that~$0\in\partial\{v>0\}$
and that~$v$ is non-degenerate at~$0$. Then~$\partial\{v>0\}$ cannot
be differentiable at the origin.
\end{theorem}

The paper is organized as follows: 
in Section \ref{sec-exist} we establish the existence of a 
strong solution of \eqref{pde-000} in the unit ball $B_1$
and thus prove Theorem~\ref{EX:TH}. Next, using a dyadic 
scaling argument, 
we prove that a solution~$v(x)$ grows away from the 
free boundary $\fb v$ as $[\hbox{dist}(x, \fb v)]^\beta$.
This is contained in Section \ref{sec-growth}, which
will provide the proof of Theorem~\ref{prop-dyadic:TH}.
Our main technical tool, which
is the
uniform integral bound in Theorem~\ref{PRE:thm-2D}, is 
established in Section~\ref{sec-Spruck}. 
To this goal, we use some
computations based on the ideas of Joel Spruck \cite{S83}.
Section~\ref{sec-Spruck} also contains
the proof of Theorem~\ref{thm-2D}, which fully relies on the integral
estimate in~\eqref{098io1}.
Finally, in Section~\ref{sec-global} we show that the free boundary cannot be regular at the 
free boundary points where $a_{ij}$ suffers a 
discontinuity satisfying $\bf (H3)$, thus completing the proof of
our main result in Theorem~\ref{DISCO}.

\section{Existence of solutions}\label{sec-exist}

In this section, we give the proof of the existence result
in Theorem~\ref{EX:TH}.

\begin{proof}[Proof of Theorem~\ref{EX:TH}]
The proof is based on a classical penalization argument. The case of the obstacle problem,
corresponding to~$p=0$, is treated in~\cite{Blan}.
Our proof is similar, but we will sketch it for the reader's convenience since unlike \cite{Blan} our coefficients are not in VMO. In fact, for our case $p\in(0,1)$
the proof is shorter since for $p>0$ the penalization function~$\phi_\e$ 
(see below) is continuous at the origin.
Hence,
by a customary  compactness 
argument, we deduce that the limit of the penalized problem is a solution of \eqref{pde-000} a.e. 
Therefore, we only need to establish uniform estimates for the penalized problem \eqref{penal-hav}.
The details of the proof go as follows.
\medskip 

Let $\eta\in C_0^\infty(\R^n)$ such that ${\rm{\supp }}\,\eta\subset B_1$, 
$\eta\ge 0$ and $\int_{B_1} \eta=1$. 
Let $\eta_\ep(x)=\ep^{-n}\eta(x/\ep)$. Then 
$\eta_\ep$ is a standard mollifier. Set $a_{ij}^\ep:=a_{ij}*\eta_\ep$
and~$g_\ep:=g*\eta_\ep$, where~$g$ is as in the statement of
Theorem~\ref{EX:TH}. Furthermore, let
$\phi_\ep:\R\to\R$ be a family of functions with the following properties
\begin{eqnarray*}
&& 0\le \phi_\ep(s)\le 1,\\
&& \phi_\ep(s)=0 \quad \hbox{if}\ s\le 0, \\
&& \phi_\ep(s)=s^{p}\quad \hbox{if} \ s\ge \ep,\\
&& \phi_\ep(s)\ \hbox{is monotone increasing},\\
{\mbox{and}}&& \phi_\ep\in C^\infty(\R).
\end{eqnarray*}
Then, there exists a classical solution $v^\ep$ to the following 
Dirichlet problem 
$$ %\label{ep-pde}
\left\{
\begin{array}{lll}
a_{ij}^\ep(x)\p{ij}v^\ep(x)=\phi_\ep(v^\ep(x)) \ \hbox{in}\ B_1,\\
v^\ep(x)=g_\ep(x)\ \hbox{on}\ \partial B_1.
\end{array}
\right.
$$
Now, for every~$t\in[0,1]$,
we consider the penalized problem 
\be\label{ep-pde}
\left\{
\begin{array}{lll}
a_{ij}^\ep(x)\p{ij}v^\ep_t(x)=t\phi_\ep(v^\ep_t(x)) \ \hbox{in}\ B_1,\\
v^\ep_t(x)=g_\ep(x)\ \hbox{on}\ \partial B_1.
\end{array}
\right.
\ee
Here, the subscript~$t$ is just a parameter, and does not denote the time derivative.
We set
$$ I:=\{t\in[0,1] \ \hbox{s.t.}\ \eqref{ep-pde}\ \hbox{has a solution}\}$$
and we
claim that
\begin{equation} \label{C1}   
{\mbox{$I$ is open}}.\end{equation} 
Note that $a_{ij}^\ep(x)\p{ij}v^\ep_t(x)\ge 0$, hence by the maximum 
principle
$| v^\ep_t(x)|\le C+\|g_\ep\|_\infty$, for some~$C>0$.
For any~$t\in[0,1]$,
we consider the operator 
$A_tu:=a_{ij}u_{ij}-t\phi_\ep(u)$.
Then the Fr\'echet derivative of $A_t$ is 
\[DA_th=a_{ij}h_{ij}-t\phi_\ep'(u)h.\]
%% because we have 
%% \[{A_t(u+h)-A_t(u)}={a_{ij}h_{ij}-t\phi_\ep'(u)h+o(h)}.\]
Thus the derivative operator has the form 
\[DA_th=a_{ij}h_{ij}+tc(x)h, \quad \hbox{with}\ c(x)\le 0\]
since, by construction, $\phi_\ep$ is monotone increasing. 
Applying the Schauder theory in Chapter~6 of~\cite{GT}, we conclude that for any 
$f\in C^\alpha$
and~$g\in C^{2, \alpha}(\overline{B_1})$
there exists a solution~$w^\ep$ of
\be \left\{
\begin{array}{lll}
DA_t w^\ep=f \ \hbox{in}\ B_1,\\
w^\ep(x)=g_\ep(x)\ \hbox{on}\ 
\partial B_1.
\end{array}\right.
\ee
This
implies that $DA_t:C^{2, \alpha}(\overline{B_1})\to 
C^{2, \alpha}(\overline{B_1})\oplus C^{\alpha}(\overline{\partial B_1})$ 
is surjective. By the maximum principle (recall that~$c(x)=-\phi_\ep'(v^\ep_t)\le 0$)
$DA_t$ is also injective. 
Therefore, $DA_t$ is invertible,
which establishes~\eqref{C1}.
\medskip

Now we show that 
\begin{equation} \label{C2}   
{\mbox{$I$ is closed}}.\end{equation}
To this aim,
we first observe that,
from the Sobolev embedding, we have that
$\|v^\ep_t\|_{C^{1, \alpha}}\lesssim 
\|v^\ep_t\|_{W^{2, q}}$. Consequently, 
applying the Schauder estimates in Chapter~6 of~\cite{GT}, we obtain that 
\[\|v^\ep_t\|_{C^{4, \alpha}}\le C(\ep),\]
for some~$C(\ep)>0$,
independently of $t$. Thus if $I\ni t_k\to t_0$ then from Arzela-Ascoli	theorem it follows that 
$v^\ep_{t_k}\to v^\ep_{t_0}$ in $C^{4, \alpha}(\overline{B_1})$ and $v_{t_0}^\ep$ solves the 
corresponding problem \eqref{ep-pde},
thus proving~\eqref{C2}.\medskip

Now, 
from~\eqref{C1} and~\eqref{C2},
we deduce that a solution of~\eqref{ep-pde}
exists for all~$t\in[0,1]$.
By Theorem~4.2 in~\cite{VMO}, we have that
\be\label{penal-hav}
\|v^\ep_t\|_{W^{2, q}(B_1)}\le C, \quad {\mbox{for some }}q>1,
\ee
uniformly in $\ep$
because $a_{ij}^\ep$ verifies $\bf (H1)$-$\bf (H3)$.
\end{proof}

\section{Optimal growth from the free boundary}\label{sec-growth}
Let $x_0\in \fb v\cap B_1$ and consider the scaled function 
$$v_r(x):= \frac{v(x_0+rx)}{r^\beta}, \qquad r>0.$$
We remark that if the inequality 
\begin{equation}\label{C3}
v(x)\le C|x-x_0|^\beta\end{equation}  
holds in some neighborhood of 
$x_0$, for some constant $C>0$ and~$\beta$ as in~\eqref{C0},
then $v_r$ is uniformly bounded as $r\to 0$.

So, we show that the growth control in~\eqref{C3}
is indeed satisfied for bounded solutions of~\eqref{pde-000}.
The result that we have is the following:

\begin{prop}\label{prop-dyadic}
Let $u\ge 0$ be a weak solution of \eqref{pde-000} in $B_1$ such that 
\[0\le u(x)\le M\]
for some constant $M>0$. Then there exists a constant $C>0$ such that 
for each $x\in B_{\frac12}\cap \fb v$ there holds
\be
S(k+1, x)\le \max\left\{
\frac{CM}{2^{\beta k}}, \frac12S(k, x)
\right\} 
\ee
where $S(k, x):=\sup_{B_{2^{-k}}(x)}u$.
\end{prop}

\begin{remark}\label{RE:dyadic}
It is well known that the estimate in Proposition \ref{prop-dyadic} implies the 
desired growth rate
in~\eqref{C3}. 
\end{remark}

\begin{proof}[Proof of Proposition~\ref{prop-dyadic}]
We use a dyadic scaling argument. Suppose that the claim in Proposition \ref{prop-dyadic} 
fails, then there exists a sequence of integers $k_i$, and points $x_i\in B_{\frac12}\cap \fb{v}$ such that 
\be\label{hav-1}
S(k_i+1, x_i)>\max\left\{
\frac{i M}{2^{\beta k_i}}, \frac12S(k_i, x_i)
\right\}.
\ee
We introduce the scaled functions
\be
u_i(x):=\frac{v(x_i+2^{-k_i}x)}{S(k_i+1)},
\ee
where~$S(\cdot)$ is a short notation for~$
S(\cdot,x_i)$.
Then, we have that 
\be\label{hav-nondeg}
\sup_{B_{\frac12}}u_i=
\frac{ \sup_{B_{2^{-k_i+1}}(x_i)}u }{S(k_i+1)}=1
,\ee
and, from \eqref{hav-1},
\be\label{hav-sup}
\sup_{B_1}u_i
=
\frac{ \sup_{B_{2^{-k_i}}(x_i)}u }{S(k_i+1)}
=\frac{S(k_i)}{S(k_i+1)}
\le 2.
\ee
Furthermore, setting $r_i:=2^{-k_i}$, by a direct computation we see that 
\begin{eqnarray*}
\sum_{l,m}a_{lm}(x_i+xr_i)\p{lm} u_i(x)
&=&\frac{2^{-2k_i}}{S(k_i+1)}u^{p}(x_i+xr_i)\\
&=&\frac{2^{-2k_i}S^{p}(k_i+1)}{
S(k_i+1)}u^{p}_i(x)\\
&=&
\frac{1}{2^{2k_i}S^{1-p}(k_i+1)}u_i^p(x).
\end{eqnarray*}
Notice also
that \eqref{hav-1} and~\eqref{C0} yield that 
\begin{eqnarray*}
iM&\le& 2^{\beta k_i}S(k_i+1)\\
&=&\left(2^{2k_i} S^{\frac2\beta}(k_i+1)\right)^{\frac\beta2}\\
&=&\left(2^{2k_i} S^{1-p}(k_i+1)\right)^{\frac\beta2}.
\end{eqnarray*}
Consequently, recalling~\eqref{hav-sup},
we have that
\begin{equation}\label{3.5BIS}
0\le \sum_{l,m}a_{lm}(x_i+ xr_i)
\p{lm} u_i(x)
\le \frac{u^p_i(x)}{(k_iM)^{\frac2\beta}}
\le
\frac{2^p}{(k_iM)^{\frac2\beta}}\to 0
\quad \hbox{as}\quad i\to\infty.
\end{equation}
Let us define the sequence of 
matrices $A^i_{lm}(x):=a_{lm}(x_i+r_ix)$.
Then $A^i(x)$ satisfies $\bf (H1)$. Observe that 
the change of variables $\xi=x_i+r_ix$  implies 
\[\fint_{B_r(z)}A^i_{lm}=\fint_{B_{rr_i}(x_i+r_iz)}a_{lm}.\]
Recalling that $x\in B_{\frac12}$, we see that 
\be\label{hav-small} 
\sup_{0<r\leq R}\sup_{z\in \R^n}\fint_{B_r(z)}\left|A^i_{lm}(x)-\fint_{B_r(z)}A_{lm}^i\right|dx = 
\sup_{0<r\leq Rr_i}\sup_{y\in \R^n}\fint_{B_r(y)}\left|a_{lm}(\xi)-\fint_{B_r(y)}a_{lm}\right|d\xi
\le \delta(Rr_i),
\ee 
implying that $\bf (H2)$ is also satisfied for the matrices $A^i$.
 
Furthermore, in light of~\eqref{3.5BIS},
we see that~$u_i$ solves the inequality 
\be\label{hav-pde}
\left|\sum_{l,m}A^i_{lm}(x)\p{lm}u_i(x)\right|\le \frac{2^p}{(k_iM)^{\frac2\beta}}\to 0.
\ee
{F}rom \eqref{hav-sup}, 
\eqref{hav-small} and \eqref{hav-pde}
it follows that we can apply Theorem 4.1 
in~\cite{VMO} to conclude that for any $q>1$ the following estimate holds uniformly in $i$
\be\label{hav-W2p}
\|u_i\|_{W^{2,q}(B_{\rho})}\le C(\rho, q)
\ee
where $B_\rho$ is a fixed ball but with arbitrary radius $\rho>0$.
Consequently, the sequence of strong solutions $\{u_i\}$ is bounded in 
$W^{2, q}_{loc}\cap L^\infty$. {F}rom Krylov-Safonov theorem it follows that 
for a subsequence, still denoted by $u_i$,
we have that~$u_i\to u$ in $B_{\frac34}$ uniformly. Thus 
$u_i(0)=0$ and \eqref{hav-nondeg} 
translates to the limit function $u$, namely we have  
\[u(0)=0, \quad u(x)\ge 0, \quad \sup_{B_{\frac12}} u=1, \quad \sup_{B_1} u\le 2.\]
On the other hand $A^i\to A^0$ a.e. and $A^0$ satisfies $\bf (H1)$-$\bf (H3)$. In particular, 
$A^0_{lm}u_{lm}=0$ a.e. Hence,  $u(0)=0$ and the strong maximum principle imply that $u\equiv 0$  which is in contradiction with  $\sup_{B_{\frac12}} u=1$ and the proof is complete.
\end{proof}

{F}rom Proposition \ref{prop-dyadic} and Remark~\ref{RE:dyadic}
we obtain Theorem~\ref{prop-dyadic:TH}, as desired.

\section{Blow-up sequences and homogeneity}\label{sec-Spruck}

We want to show that, using a technique
invented by J.
Spruck in~\cite{S83},
at the non degenerate free boundary points 
the blow-up is a homogeneous function of degree $\beta$.
For
a sequence of positive numbers $r_k\to 0$ and $x_0\in \fb v$, we consider  
the blow-up sequence 
\begin{equation}\label{COMPA0}v_{r_k}(x):=\frac{v(x_0+r_k x)}{r_k^\beta}.\end{equation}
{F}rom Theorem \ref{prop-dyadic:TH}
we know that the sequence $\{v_{r_k}\}$ is bounded and solves 
equation \eqref{pde-000} with $a_{ij}$ satisfying $\bf(H1)$-$\bf(H3)$. Thus, applying Theorem 
4.1 in~\cite{VMO}, we conclude that 
$\{v_{r_k}\}$ is locally uniformly bounded in $W^{2, q}$ for any 
$q>1$.
Then a customary compactness argument implies 
that there exists a subsequence $\{v_{k_i}\}$ 
and $v_0$, such that 
\begin{equation}\label{COMPA}
{\mbox{$v_{k_i}\to v_0$ in $C^1_{loc}(\R^n)$.}}\end{equation} The function~$v_0$ is called a blow-up limit at $x_0$.

\subsection{$2$D problems}

As customary, it is often useful to
write solutions of partial differential equations
in polar coordinates. In our case, we
have the following result:

\begin{lemma}\label{LEL:rot2d}
Let~$\L$ be as in~\eqref{pde-000},
with~$a_{ij}$ as in~\eqref{aijdef}. Then
\begin{equation}\label{pde-pol}
\L v=
\p{rr}v+\frac1r\p r v+\frac1{r^2}\p {\th\th} v+\e\left(\frac{\p {\th\th}v}{r^2}+\frac{\p rv}r\right).
\end{equation}
\end{lemma}

\begin{proof}
We will use polar coordinates $r$, $\theta $ and rewrite  the partial derivatives as follows
\begin{eqnarray}
\p {x_1}=\cos\theta\p r-\frac{\sin\theta}r\p \theta, \quad \p {x_2}=\sin \theta \p r+\frac{\cos\theta}r\p\theta.
\end{eqnarray}
By a straightforward
computation we have that 
\begin{equation}\label{blya}\begin{split}
2x_1x_2\partial_{12}v=\;&2r^2\cos\theta\sin\theta\Big\{
\sin \th\cos\th \p{rr}v-\frac{\sin\th\cos\th}{r}\p rv
+\frac{\cos^2\th-\sin^2\th}r\p{\th r}v\\\nonumber
&\qquad\qquad\qquad 
+\frac{\sin^2\th-\cos^2\th}{r^2}\p\th v
-\frac{\sin\th\cos\th}{r^2}\p{\th\th}v
\Big\},\\
x_2^2\p{11}v=\;&r^2\sin^2\th \left\{
\cos^2\th\p{rr}v+\frac{\sin^2\th}{r}\p rv
-\frac{2\sin\th\cos\th}r\p{r\th}v
+\frac{2\sin\th\cos\th}{r^2}\p\th v+\frac{\sin^2\th}{r^2}\p{\th \th} v
\right\},\\
x_1^2\p{22}v=\;&r^2\cos^2\th\left\{
\sin^2\th\p{rr}v+\frac{\cos^2\th}r \p rv+\frac{2\sin\th\cos\th}r \p{r\th}v-
\frac{2\sin\th\cos\th}{r^2}\p\th v+\frac{\cos^2\th}{r^2}\p{\th\th}v
\right\}.
\end{split}\end{equation}
Combining these three identities and recognizing the terms we get that
\begin{eqnarray*}
&&\frac1\e\left(\L v-\Delta v\right)\\&=&
\p{r\th}\frac{2\sin\th\cos\th}r\left[\cos^2\th
-\sin^2\th-\cos^2\th+\sin^2\th\right]
+\p{\th\th}v\left[\cos^4\th+\sin^4\th
+2\sin^2\th\cos^2\th\right]\\
&&+\p r v \frac1r\left[\cos^4\th+\sin^4\th
+2\sin^2\th\cos^2\th\right]+
\p\th v\frac{2\sin\th\cos\th}r
\left[\sin^2\th-\cos^2\th-\sin^2\th+\cos^2\th\right]\\
&=&\frac{\p {\th\th}v}{r^2}+\frac{\p rv}r.
\end{eqnarray*}
Using this and the standard representation of the Laplacian
in polar coordinates, the desired result follows.
\end{proof}

With this, we are in position of proving Theorem~\ref{PRE:thm-2D}.

\begin{proof}[Proof of Theorem~\ref{PRE:thm-2D}]
We let $r:=e^{-t}$ and $w(t,\th):=\frac{v(r, \th)}{r^{\beta}}$.
Then we have 
\begin{eqnarray*}&&
\p \th w=\frac{\p\th v}{r^\beta},\\ &&
\p{\th\th} w=\frac{\p{\th\th} v}{r^\beta},\\ {\mbox{and }}&&
\p t w=-\frac{\p r v}{r^{\beta-1}}+\beta w.
\end{eqnarray*}
Plugging this into \eqref{pde-pol} we infer that
\begin{eqnarray*}
 r^{\beta-2}((\p{tt}w-\p tw-\beta(\beta-1)w)+(\beta w-\p tw)+\p{\th\th}w)+\e\left(r^{\beta-2}\p{\th\th}w+r^{\beta-2}[\beta w-\p tw]\right)=r^{-\beta p}w^p
.\end{eqnarray*}
This, after recalling that $\beta-2=-p\beta$, yields that
\begin{equation}\label{erjghgbv:0}
 I_1+\e I_2=w^p,
\end{equation}
where 
\[I_1:=\p{tt}w-2\p t w+\p{\th\th}w-\beta(\beta-2)w \quad 
{\mbox{ and }}\quad
I_2:=\p{\th\th}w+\beta w-\p tw.\]
Next, we multiply both sides of equation~\eqref{erjghgbv:0} 
by~$\partial_t w$
and we integrate first over the unit circle and then in
the interval $[T_1, T_2]$ to get that
\begin{equation}\label{erjghgbv}
\int_{T_1}^{T_2}\int_{\S^1}I_1\,\partial_t w 
+\e\int_{T_1}^{T_2}\int_{\S^1}I_2\,\partial_t w=
\int_{T_1}^{T_2}\int_{\S^1}w^p \,\partial_t w.
\end{equation}
Now we observe that
\begin{equation}\label{erjghgbv:1}\begin{split}
\int_{T_1}^{T_2}\int_{\S^1}I_2\,\partial_t w=\;&-\int_{T_1}^{T_2}
\int_{\S^1}(\p t w)^2+ \beta\int_{T_1}^{T_2}
\int_{\S^1}w\,\p tw+ \int_{T_1}^{T_2}
\int_{\S^1}\p{\th\th}w\,\p t w\\
=\;&-\int_{T_1}^{T_2}\int_{\S^1}
(\p t w)^2+\beta\int_{\S^1}\frac{w^2}2
\Bigg|_{T_1}^{T_2}-\int_{T_1}^{T_2}\int_{\S^1}\p \th w\,\p{r\th} w
\\
=\;&-\int_{T_1}^{T_2}\int_{\S^1}(\p t w)^2+
\beta\int_{\S^1}\frac{w^2}2\Bigg|_{T_1}^{T_2}
-\int_{\S^1}\frac{(\p{\th} w)^2}{2}\Bigg|_{T_1}^{T_2}.
\end{split}\end{equation}
Similarly,
\begin{equation}\label{erjghgbv:2}
\int_{T_1}^{T_2}\int_{\S^1}I_1\,\partial_t w=
-2\int_{T_1}^{T_2}\int_{\S^1}(\p t w)^2+
\int_{\S^1}\frac{w^2}2\Bigg|_{T_1}^{T_2}-
\int_{\S^1}\frac{(\p{\th} w)^2}{2}\Bigg|_{T_1}^{T_2}-
\frac{\beta(\beta-2)}2\int_{\S^1}\frac{w^2}2\Bigg|_{T_1}^{T_2}.
\end{equation}
Moreover, 
$$\int_{T_1}^{T_2}\int_{\S^1}w^p\,\p tw 
=\int_{\S^1}\frac1{p+1}w^{p+1}\Bigg|_{T_1}^{T_2}
$$
So, plugging this, \eqref{erjghgbv:1} and~\eqref{erjghgbv:2} 
into~\eqref{erjghgbv}, we obtain that
\begin{equation*}\begin{split}&
(\e+2 )\int_{T_1}^{T_2}\int_{\S^1}(\p t w)^2=
\e\left\{\beta\int_{\S^1}\frac{w^2}2\Bigg|_{T_1}^{T_2}-
\int_{\S^1}\frac{(\p{\th} w)^2}{2}\Bigg|_{T_1}^{T_2}\right\}
-\int_{\S^1}\frac1{p+1}w^{p+1}\Bigg|_{T_1}^{T_2}
\\&\qquad\qquad
+
\int_{\S^1}\frac{w^2}2\Bigg|_{T_1}^{T_2}-\int_{\S^1}\frac{(\p{\th} w)^2}{2}\Bigg|_{T_1}^{T_2}-\frac{\beta(\beta-2)}2\int_{\S^1}\frac{w^2}2\Bigg|_{T_1}^{T_2}.
\end{split}\end{equation*}
Since $\p t w=-\frac{\p r v}{r^{\beta-1}}+\beta \frac{v}{r^\beta}
$, the last inequality then reads
\[
\int_{T_1}^{T_2}\int_{\S^1}\left(\beta \frac{v}{r^\beta}  
-\frac{\p r v}{r^{\beta-1}}\right)^2\, d\th\,dt\le \tilde C,
\]
where $\tilde C$ depends only on the the constant 
$M$ in the growth estimate  $v(x)\le M |x|^\beta$, see 
Theorem \ref{prop-dyadic:TH}.
Since $T_1$ and~$ T_2$ are arbitrary,
by the change of variable~$r:=e^{-t}$ we obtain that
$$ \int_{0}^{1/2}\int_{\S^1}\left(\beta\, \frac{v(r,\theta)}{r^\beta}  
-\frac{\p r v(r,\theta)}{r^{\beta-1}}\right)^2\,\frac{dr\, d\th}{r}
\le \tilde C.$$
This implies the desired result via polar coordinates.
\end{proof}

{F}rom Theorem~\ref{PRE:thm-2D}, we obtain the homogeneity
of the blow-up sequences, according to Theorem~\ref{thm-2D}:

\begin{proof}[Proof of Theorem~\ref{thm-2D}]
By~\eqref{098io1}, a change of variable~$x=\rho y$
gives that
$$ \int_{B_{\frac1{2\rho}}} \left(\beta\, \frac{v_\rho(y)}{|y|^\beta}  
-\frac{\p r v_\rho( y)}{|y|^{\beta-1}}\right)^2\,\frac{dy}{|y|^2}
\le \tilde C,$$
where the notation in~\eqref{COMPA0} has been used.
This and~\eqref{COMPA} imply that
$$ \int_{\R^n} \left(\beta\, \frac{v_0(y)}{|y|^\beta}  
-\frac{\p r v_0(y)}{|y|^{\beta-1}}\right)^2\,\frac{dy}{|y|^2}
\le \tilde C,$$
and so
$$ \beta\, \frac{v_0(y)}{|y|^\beta}  
=\frac{\p r v_0(y)}{|y|^{\beta-1}},$$ for any~$y\in\R^n$,
which implies the desired result (see e.g. Lemma~4.2
in~\cite{DSV15}).
\end{proof}

\subsection{$n$-dimensional problems}
For the sake of completeness,
we consider now a multidimensional 
model.
We take
\begin{equation}\label{def aij dim n}
a_{ij}(x):=\delta_{ij}+\e x_ix_j|x|^{-2}.\end{equation} Notice that the 
hypotheses in~{\bf (H1)-(H3)} 
are satisfied for sufficiently small $|\e|$. 

We extend Theorem \ref{thm-2D} to this case. To this aim, let us 
switch to polar coordinates and define
\[
\begin{array}{lll}
x_1=r\cos\theta_1\\
\ \vdots\\
x_k=r\sin\th_1\sin\th_2\dots\sin\th_{k-1}\cos\th_k\\
\ \vdots\\
x_n=r\sin\th_1\sin\th_2\dots\sin\th_n,
\end{array}
\]
where $0\le\th_k\le \pi$, with~$
k=1, \dots, n-2,$ and $ -\pi\le 
\th_{n-1}\le \pi$.
In this setting, the analogue of Lemma~\ref{LEL:rot2d}
goes as follows:

\begin{lemma}
Let~$\L$ be as in~\eqref{pde-000},
with~$a_{ij}$ as in~\eqref{def aij dim n}. 
Assume that~$x$ lies on the $x_1$ 
axis.
Then 
\begin{equation} \label{CONS}\L v=
(1+\e)\p{rr}v+\frac1r\p r v
+\frac1{r^2}\p {\th\th} v.
\end{equation}
\end{lemma}

\begin{proof} {F}rom the chain rule, we have that 
\be
\frac{\partial v}{\partial x_1}=\frac{\partial v}{\partial r}\frac{\partial r}{\partial x_1}+\frac{\partial v}{\partial \th_1}\frac{\partial \th_1}{\partial x_1}
=\frac{\partial v}{\partial r}\cos\th_1-\frac{\sin \th_1}r\frac{\partial v}{\partial \th_1}
.\ee
Hence, proceeding as in \eqref{blya}, 
and using $\th=0$ to set the point on the~$x_1$ axis,
we get that
$$
\frac{x_1^2}{|x|^2}\p{11} v=\p{rr}v,
$$
which gives the desired result.
\end{proof}

In this setting, the analogue of Theorem~\ref{thm-2D}
is the following:

\begin{theorem}
Let $v$ be a strong solution of \eqref{pde-000} in $B_1\subset\R^n$ with $a_{ij}$
as in~\eqref{def aij dim n}.
Assume that~$0\in \fb v$ and $v$ is non-degenerate at $0$. Then any blow-up sequence at $0$ has a converging subsequence 
such that the limit is a homogeneous function of degree $\beta=\frac2{1-p}$. 
\end{theorem}

\begin{proof}
We use the change of variables $r=e^{-t}, (\th_1, \dots, \th_{n-1})\in \mathbb S^{n-1}$, where 
$\mathbb S^{n-1}$ is the unit sphere in $\R^n$. Hence, for the
function $w(t,\th)=\frac{v(r, \th)}{r^{\beta}}$,
making use of~\eqref{CONS},
equation~\eqref{pde-000} can be rewritten as
\[
(1+\e)(\p{tt}w-\p tw-\beta(\beta-1)w)+(\beta w-\p tw+\Delta_{\th\th}w)=w^p
,\]
where $\Delta_{\th\th}$ is the 
Laplace-Beltrami operator 
on the unit sphere.
Thus, repeating the integration by parts as in the proof of 
Theorem \ref{PRE:thm-2D}
and the scaling argument in the proof of Theorem~\ref{thm-2D}, the desired result follows. 
\end{proof}

\section{Global homogeneous solutions}\label{sec-global}

In this section, we would like to classify the global solutions 
of~\eqref{pde-000} in the plane
in the homogeneous setting for the case of the obstacle problem.

\begin{theorem} \label{n2class}
Let~$n=2$, $\L$ be as in~\eqref{pde-000}
and~$a_{ij}$ as in~\eqref{aijdef}. Let~$v$ be 
a solution of~\eqref{pde-000} in~$\R^2$ with~$p=0$
which is homogeneous of degree~$2$. Assume that~$0\in\partial\{v>0\}$
and that~$\partial\{v>0\}$ is differentiable
at the origin.
Then~$\e$ in~$a_{ij}$ needs to be equal to~$0$
(and thus~$a_{ij}=\delta_{ij}$).
\end{theorem}

\begin{proof} We first make a general calculation
valid for all~$p\in[0,1)$.
Let $v(x)=r^\beta g(\th)$. 
We suppose (up to a rotation) 
that the arc $(0, \alpha)$ is a component of the positivity set of $g$.
In this way,
\begin{equation}\label{GG}g(0)=g(\alpha)=0.
\end{equation}
We let~$x_0:=(1,0)$. {F}rom Remark~\ref{RE:dyadic},
we know that~\eqref{C3}
is satisfied, and thus there exists~$M>0$
such that
$$ M\,|x-x_0|^\beta \ge v(x)=r^\beta g(\theta) = r^\beta\big|g(\theta)-g(0)\big|.
$$
For a small~$t>0$, we evaluate this formula at the point~$x_t:=(1,t)$,
which corresponds in polar coordinate to~$r_t:=\sqrt{1+t^2}$ and~$\theta_t =\arctan t$.
In this way, we obtain that
$$ M t^\beta \ge (1+t^2)^{\frac\beta2} \big|g(\arctan t)-g(0)\big|.$$
So, dividing by~$t$ and sending~$t\to0$, using the fact that~$\beta>1$,
$$ 0 \ge \lim_{t\to0} \left|\frac{g(\arctan t)-g(0)}{t}\right|=
|g'(0)|$$
and so
\begin{equation}\label{GP0}
g'(0)=0.\end{equation}
Furthermore, from \eqref{pde-pol},
\[
\beta(\beta-1)g+\beta(1+\e)g+(1+\e)g''=g^p,
\]{or equivalently}
\[
\beta(\beta+\e)g+(1+\e)g''=g^p.
\]
Multiplying both sides by $g'$ and integrating yields
\be\label{blyaa}
(1+\e)[g']^2+\beta(\e+\beta) g^2+C_o=\frac2{p+1}g^{p+1}\ee
where $C_o\in\R$ is an arbitrary constant. 
Using~\eqref{GG} and~\eqref{GP0}, we have that~$g(0)=0=g'(0)$, which gives that $C_o=0$.
Moreover 
\[g^2\left(\frac{g^{p-1}}{p+1}-\frac{\beta(\beta+\e)}2\right)\ge 0.\]
Consequently, solving \eqref{blyaa} we obtain  
\[g'=\pm\frac1{\sqrt{1+\e}}\sqrt{\frac{2}{p+1}g^{p+1}-\beta(\e+\beta) g^2}.\]
This is a separable equation, and so we obtain
\begin{equation}\label{gu91}
\int \frac{dg}{\sqrt{\frac{2}{p+1}g^{p+1}-\beta(\e+\beta) g^2}}
=\pm \frac{1}{\sqrt{1+\e}}\int d\th+C. \end{equation}
The integrals above may be explicitly computed in terms
of hypergeometric functions for any~$p\in[0,1)$, but,
for concreteness, we now restrict ourselves to the case~$p=0$.
In this case, \eqref{gu91} becomes
\begin{equation}\label{9sudhvgsdwe}
\frac1{\sqrt 2}\int \frac{dg}{\sqrt{g-(2+\e) g^2}}=
\pm \frac{1}{\sqrt{1+\e}}\int d\th+C .\end{equation}
We now set~$a_\e:=\frac1{2(2+\e)}$ and we observe that
$$ g-(2+\e) g^2=(2+\e)(
2a_\e g- g^2) =(2+\e)(a_\e^2 -(a_\e-g)^2).$$
Hence, the substitution~$h:=(g/a_\e)-1$ in~\eqref{9sudhvgsdwe}
gives that
\begin{equation*}
\frac1{\sqrt{2(2+\e)}}\int \frac{dh}{\sqrt{
1 -h^2
}}=
\pm \frac{1}{\sqrt{1+\e}}\int d\th+C ,\end{equation*}
and so
\begin{equation}\label{78:PEREQ}
\begin{split} \frac1{\sqrt{2(2+\e)}}\,
\arcsin\frac{g-a_\e}{a_\e}&=\frac1{\sqrt{2(2+\e)}}\,
\arcsin h\\&
=\pm \frac{1}{\sqrt{1+\e}}\int d\th+C\\&=
\pm \frac{1}{\sqrt{1+\e}}\,\th+C.
\end{split}\end{equation}
Then, evaluating~\eqref{78:PEREQ} at~$\theta:=0$
and using~\eqref{GG}, we obtain that
$$ \arcsin(-1)=\arcsin\frac{g(0)-a_\e}{a_\e}=
{\sqrt{2(2+\e)}}\,C.$$
Thus, defining
$$\omega_\e:=\pm \sqrt{\frac{2(2+\e)}{1+\e}},$$ 
we rewrite~\eqref{78:PEREQ}
as
\begin{equation}\label{THIS}
\arcsin\frac{g(\theta)-a_\e}{a_\e}=\omega_\e\theta+\arcsin(-1).\end{equation}
Since~$\partial\{v>0\}$ is smooth and~$v$ homogeneous,
formula~\eqref{GG}
says that~$\alpha=k\pi$, with~$k\in\{1,2\}$.
Evaluating~\eqref{THIS}
at~$\theta:=k\pi$ and~$\theta:=0$,
using that~$g(0)=g(k\pi)=0$ (in view of~\eqref{GG}),
we obtain that
\begin{eqnarray*}
0&=&
\frac{g(k\pi)-a_\e}{a_\e}-
\frac{g(0)-a_\e}{a_\e}\\&=&
\sin\left( \omega_\e\,k\pi+\arcsin(-1) \right)-
\sin \left(\arcsin(-1)\right) \\&=&
-\cos \left( \omega_\e\,k\pi\right)
+1
\end{eqnarray*}
and therefore~$\omega_\e\,k\pi\in 2\pi\Z$.
This gives that
$$ \pm k\,\sqrt{\frac{2(2+\e)}{1+\e}}\in2\Z,$$ and
so
$$ \sqrt{\frac{2(2+\e)}{1+\e}}\in\Z,$$
which, for small~$\e$, only holds when~$\e=0$.
\end{proof}

\begin{remark} {F}rom~\eqref{THIS},
one can also construct a homogeneous solution~$v\ge0$
of the obstacle problem~$\L v=1$ in~$\{v>0\}$, with~$\L$
as in~\eqref{pde-000} and~$a_{ij}$
as in~\eqref{aijdef},
whose free
boundary is a cone, namely, in polar coordinates, one can take~$v=v(r,\theta)=
r^2\,g(\theta)$, with
$$ g(\theta)=\left\{\begin{matrix}
a_\e\,\big(1
-\cos(\omega_\e\theta)\big)
& {\mbox{ if }}\theta\in\left( 0,\frac{2\pi}{\omega_\e}\right),\\
0&{\mbox{ otherwise,}}
\end{matrix}\right. $$
where~$a_\e:=\frac1{2(2+\e)}$ 
and~$\omega_\e:= \sqrt{\frac{2(2+\e)}{1+\e}}<2$
when~$\e>0$ (respectively,~$\omega_\e:= \sqrt{\frac{2(2+\e)}{1+\e}}>2$
when~$\e<0$),
see Figure~\ref{FF1}. Notice in particular, that the singular cone
of the free boundary can be either obtuse or acute, according to the cases~$\e>0$
and~$\e<0$.

\begin{figure}[h]
    \centering
    \includegraphics[width=6.8cm]{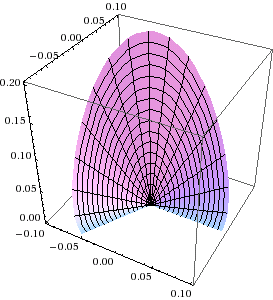}\qquad
    \includegraphics[width=6.8cm]{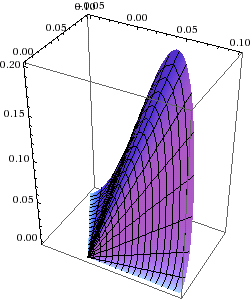}
    \caption{Examples of homogeneous solutions of the obstacle
problem with obtuse/acute singular free boundary.}
    \label{FF1}
\end{figure}

\end{remark}

Theorem~\ref{DISCO} says that this example is somehow
``typical'', namely if the free boundary of~\eqref{pde-000} meets 
the discontinuity points of the coefficients~$a_{ij}$ in a non-degenerate way,
then a singularity occurs. The proof of this fact is based on
Theorem~\ref{n2class}, and the details go as follows:

\begin{proof}[Proof of Theorem~\ref{DISCO}] Assume by contradiction that~$\partial\{v>0\}$
can be written as a differentiable graph near the origin:
say, up to a rotation, that $\{v>0\}$ coincides 
with~$\{x_2<\varphi(x_1)\}$
near the origin, with~$\varphi$ differentiable, $\varphi(0)=0$
and~$\varphi'(0)=0$.
We consider the blow-up sequence~$v_{r_k}$ as
in~\eqref{COMPA0} (with~$x_0=0$). {F}rom the discussion at the beginning
of Section~\ref{sec-Spruck}, we know that, for a suitable infinitesimal
sequence~$r_k$, it holds that~$v_{r_k}$ approaches
a global solution~$v_0$.
Near the origin, we have that~$\partial\{v_{r_k}>0\}$
coincides with~$\left\{x_2<\frac{\varphi(r_k x_1)}{r_k}\right\}$.
Using this and the fact that~$\varphi(r_k x_1)=o(r_k x_1)$,
we thus obtain that~$\partial\{v_{0}>0\}$
near the origin coincides with~$\left\{x_2<0\right\}$.
Also, from
Theorem~\ref{thm-2D}, we know that~$v_0$ is homogeneous of degree~$2$.
These considerations and Theorem~\ref{n2class}
imply that~$\e=0$, against our assumptions.
\end{proof}

%\vfill


\begin{thebibliography}{FMRT01}

\bibitem[AP86]{AP86}
H. W. Alt, D. Phillips.
{\em A free boundary problem for semilinear elliptic equations}. 
J. Reine Angew. Math. 368 (1986), 63--107. 

\bibitem[ALT16]{ALT16}
D. J. Ara{\'u}jo, R. Leit{\~a}o, E. V. Teixeira. 
{\em Infinity Laplacian equation with strong absorptions}.
J. Funct. Anal. 270 (2016), no. 6, 2249--2267.

\bibitem[BT14]{Blan}
I. Blank, K. Teka.
{\em The Caffarelli alternative in measure for the nondivergence 
form elliptic obstacle problem with principal coefficients in VMO}.
Comm. Partial Differential Equations 39 (2014), no. 2, 321--353.

\bibitem[C08]{C08}
X. Cabr\'e. {\em
Elliptic PDE's in probability and geometry: 
symmetry and regularity of solutions}. 
Discrete Contin. Dyn. Syst. 20 (2008), no. 3, 425--457. 

\bibitem[C77]{C77}
L. A. Caffarelli, 
{\em The regularity of free boundaries in higher dimensions}.
Acta Math. 139 (1977), no. 3-4, 155--184.

\bibitem[CS05]{CS-LIB}
L. Caffarelli, S. Salsa. {\em 
A geometric approach to free boundary problems}.
Providence: American Mathematical Society, 2005.

\bibitem[CFL93]{VMO}
F. Chiarenza, M. Frasca, P. Longo.
{\em $W^{2,p}$-solvability of the Dirichlet problem for
nondivergence elliptic equations with VMO coefficients}.
Trans. Amer. Math. Soc. 336 (1993), no. 2, 841--853. 

\bibitem[DSV15]{DSV15}
S. Dipierro, O. Savin, E. Valdinoci. {\em
A nonlocal free boundary problem}. 
SIAM J. Math. Anal. 47 (2015), no. 6, 4559--4605. 

\bibitem[GT98]{GT} D. Gilbarg, N. S. Trudinger.
{\em Elliptic Partial Differential Equations of Second Order}.
New York: Springer-Verlag, 1998.

\bibitem[K16]{K16}
A. Karakhanyan.
{\em Minimal surfaces arising in singular perturbation problems}.
Preprint, 2016. 

\bibitem[K07]{K07}
M. Kassmann. {\em
Harnack inequalities: an introduction}.
Bound. Value Probl. 2007, Art. ID 81415, 21 pp.

\bibitem[PSU12]{PSU12}
A. Petrosyan, H. Shahgholian, N. Uraltseva.
{\em Regularity of free boundaries in obstacle-type problems}. 
Providence:
American Mathematical Society, 2012.

\bibitem[PT16]{PT16}
D. dos Prazeres, E. V. Teixeira. {\em
Cavity problems in discontinuous media}.
Calc. Var. Partial Differential Equations 55 (2016), no. 1, Art. 10, 15 pp. 

\bibitem[S83]{S83}
J. Spruck.
{\em Uniqueness in a diffusion model of population biology}. 
Comm. Partial Differential Equations 8 (1983), no. 15, 1605--1620. 

\bibitem[T16]{T16}
E. V. Teixeira.
{\em Regularity for the fully nonlinear dead-core problem}. Math. Ann. 364 (2016), no. 3-4, 1121--1134.

\bibitem[T82]{T82}
N. Trudinger.
{\em Elliptic equations in non-divergence form}.
Miniconference on Partial Differential Equations. 
Proceedings of the Centre for Mathematical Analysis, v. 1. 
(Mathematical Sciences Institute, 
The Australian National University, 1982), 1--16.

\end{thebibliography}
\end{document}